\documentclass[reqno]{amsart}

\usepackage{mystyle}

\usepackage{etoolbox}

\usepackage[color=cyan!50]{todonotes} %%%
\presetkeys{todonotes}{inline}{} %%%
\usepackage{xcolor}

\usepackage[headings]{fullpage}
\usepackage{parskip}

\usepackage[utf8]{inputenc}
\usepackage[T1]{fontenc}
\usepackage[british]{babel}
\usepackage[pdfencoding=auto]{hyperref}

\usepackage{microtype}
\usepackage{booktabs}

\usepackage{amsfonts,amssymb,bbm,mathrsfs,stmaryrd}
\usepackage{amsmath}

\usepackage{mathtools}
\usepackage[noabbrev]{cleveref}
\usepackage{centernot}
\usepackage[shortlabels]{enumitem}
\usepackage{url}
\usepackage{tikz}
\usepackage{pict2e}
\usepackage{framed}

\usetikzlibrary{matrix}
\usetikzlibrary{positioning}
\usetikzlibrary{arrows}
\tikzset{> =stealth}
\usetikzlibrary{arrows.meta}
\tikzset{normalHead/.tip={Triangle[open,angle=60:4pt]},}
\tikzset{normalTail/.tip={Triangle[reversed,open,angle=60:4pt]},}

\theoremstyle{plain}
\newtheorem{theorem}{Theorem}[section]

\newtheorem{proposition}[theorem]{Proposition}

\theoremstyle{definition}
\newtheorem{definition}[theorem]{Definition}

\newtheorem{example}{Example}
\AtBeginEnvironment{example}{\vspace{11pt}\begin{leftbar}\vspace{-11pt}}
\AtEndEnvironment{example}{\end{leftbar}}

\AtBeginEnvironment{construction}{\begin{leftbar}}
\AtEndEnvironment{construction}{\end{leftbar}}

\theoremstyle{remark}

\renewcommand{\epsilon}{\varepsilon}
\renewcommand{\phi}{\varphi}

\renewcommand{\O}{\mathcal{O}}

\newcommand{\inv}{^{-1}}

\newcommand{\WSExt}{\mathrm{WSExt}}
\newcommand{\WAct}{\mathrm{WAct}}

\newcommand{\Mon}{\mathrm{\textbf{Mon}}}
\newcommand{\Act}{\mathrm{Act}}

\DeclareMathOperator{\Gl}{Gl}

\newcommand{\splitext}[6]{% (comment to avoid adding extra white space)
\tikz[baseline]{
\newdimen{\mylabelwidth}
\newdimen{\skipwidth}
\node[anchor=base] (A) {\hspace*{\dimexpr0.5pt-\pgfkeysvalueof{/pgf/inner xsep}}${#1}$};
\settowidth{\mylabelwidth}{\pgfinterruptpicture {$#2$} \endpgfinterruptpicture}
\pgfmathsetlength{\skipwidth}{max(\mylabelwidth,10pt)}
\node[right] (B) at ([xshift=\skipwidth+12pt]A.east) {${#3}$};
\settowidth{\mylabelwidth}{\pgfinterruptpicture {$#4$} \endpgfinterruptpicture}
\settowidth{\skipwidth}{\pgfinterruptpicture {$#5$} \endpgfinterruptpicture}
\pgfmathsetlength{\skipwidth}{max(\skipwidth,\mylabelwidth,10pt)}
\node[right] (C) at ([xshift=\skipwidth+12pt]B.east) {${#6}$\hspace*{\dimexpr0.5pt-\pgfkeysvalueof{/pgf/inner xsep}}};
\draw[normalTail->] (A) to node [above] {${#2}$} (B);
\draw[transform canvas={yshift=0.5ex},-normalHead] (B) to node [above] {${#4}$} (C);
\draw[transform canvas={yshift=-0.5ex},->] (C) to node [below] {${#5}$} (B);
}}

\title{$\lambda$-Semidirect Products of Inverse Monoids are Weakly Schreier Extensions}
\author[P. F. Faul]{Peter F. Faul}
\address{Department of Pure Mathematics and Statistical Sciences\\ University of Cambridge}
\email{peter@faul.io}

\date{\today}

 % gluing [sic]

\begin{document}

\maketitle

\begin{abstract}
%% TODO TODO Look at the last paragraph of the introduction for elements that are better
A split extension of monoids with kernel $k \colon N \to G$, cokernel $e \colon G \to H$ and splitting $s \colon H \to G$ is weakly Schreier if each element $g \in G$ can be written $g = k(n)se(g)$ for some $n \in N$. The characterization of weakly Schreier extensions allows them to be viewed as something akin to a weak semidirect product.
%There is a natural order structure on the set of weakly Schreier extensions between two monoids $N$ and $H$.
The motivating examples of such extensions are the Artin glueings of topological spaces and, of course, the Schreier extensions of monoids which they generalise.
In this paper we show that the $\lambda$-semidirect products of inverse monoids are also examples of weakly Schreier extensions.
%In fact this subsumes are previous example as we show that Artin glueings are in fact examples of $\lambda$-semidirect products.
The characterization of weakly Schreier extensions sheds some light on the structure of $\lambda$-semidirect products. The set of weakly Schreier extensions between two monoids comes equipped with a natural poset structure, which induces an order on the set of $\lambda$-semidirect products between two inverse monoids.
We show that Artin glueings are in fact $\lambda$-semidirect products and inspired by this identify a class of Artin-like $\lambda$-semidirect products. We show that joins exist for this special class of $\lambda$-semidirect product in the aforementioned order.
%Finally, inspired by the Artin glueing, we identify a class of Artin-like $\lambda$-semidirect products. We show that joins exist for this special class of $\lambda$-semidirect product in the aforementioned order.
%In this paper we show that Billhardt's notion of $\lambda$-semidirect products restricted to inverse monoids can naturally be viewed as weakly Schreier extensions of monoids. This allows us to leverage the characterization of weakly Schreier extensions and present a new way to view and understand $\lambda$-semidirect products. Additionally we show that Artin glueings of topological spaces can be viewed as $\lambda$-semidirect products and then generalise this notion to get a new class of well behaved $\lambda$-semidirect products. Finally we examine the poset structure the set of $\lambda$-semidirect products inherits from the poset of weakly Schreier extensions and prove that joins exist for this new class of Artin-like $\lambda$-semidirect products.
\end{abstract}

\section{Introduction}\label{sec:Introduction}
The ideas underlying the semidirect product of groups can be adapted to a number of structures. One such example is the context of semigroups wherein an action of semigroups $\alpha \colon H \times N \to N$ gives rise to a semidirect product $N \ltimes_\alpha H$ defined just as in the group case. These semidirect products have found much use in semigroup theory, for instance they provide some insight into the structure of inverse semigroups \cite{mcalister1974groups}. However, when applied naively to two inverse semigroups, this semidirect product does not in general yield an inverse semigroup. To remedy this Billhardt introduced a related notion called a $\lambda$-semidirect product \cite{billhardt1992wreath}. Given inverse semigroups $N$ and $H$ the idea is to use an action of $H$ on $N$ to equip a certain \emph{subset} of $N \times H$ (determined by the action) with a multiplication turning it into an inverse semigroup. These $\lambda$-semidirect products have since granted insight into the structure of inverse semigroups
%% Graham suggests instead focus only on the differences. He also found the 'multiplication turning it into...' confusing.

Semigroups are not the only context in which a generalization of the semidirect product can be considered. For monoids $N$ and $H$, an action $\alpha \colon H \times N \to N$ of $H$ on $N$ gives rise to a semidirect product $N \ltimes_\alpha H$ defined just as expected. In this context, work has been done to relate these semidirect products to split extensions \cite{martins2013semidirect}. Unlike in the case of groups were there is a one-to-one correspondence between semidirect and split extensions, these semidirect products of monoids correspond to only the Schreier extensions of monoids --- those split extensions $\splitext{N}{k}{G}{e}{s}{H}$ in which for each $g \in G$ there exists a unique $n \in N$ such that $g = k(n)se(g)$.
%% Graham points out that monoids are semigroups so the opening is confusing.

%% Cite weakly maltsev paper
The Schreier condition can be weakened and a new class of split extensions considered. Instead of requiring that for each $g \in G$ there is a \emph{unique} $n \in N$, we can merely require that there exists (potentially many) $n \in N$ for which $g = k(n)se(g)$. We call such an extension weakly Schreier and they were first considered in \cite{bourn2015partialMaltsev}. Their characterization in \cite{faul2019characterization} establishes that they resemble something like a weak semidirect product. The primary (non-Schreier) examples of these extensions are the Artin glueings of topological spaces, where the lattice of open sets are viewed as monoids.
%% 'resemble something like a weak semidirect product'. Grahams not happy and suggests saying something like 'these give something we call a weak semidirect product'
\subsection*{Outline}
In this paper we will show that the $\lambda$-semidirect products of inverse monoids are also examples of weakly Schreier extensions. In fact this subsumes are previous example as we show that Artin glueings are in fact examples of $\lambda$-semidirect products. The characterization of weakly Schreier extensions sheds some light on $\lambda$-semidirect products. The set of weakly Schreier extensions between two monoids comes with a natural poset structure, which induces an order on the $\lambda$-semidirect products between two inverse monoids. FThe Artin glueing leads us to define a class of Artin-like $\lambda$-semidirect products. We show that this class is closed under binary joins.
%% Fix the abstract with the ideas found here.

\section{Background}\label{sec:Background}
In this section we give an introduction to the theory of weakly Schreier extensions of monoids, summarising the core results found in \cite{faul2019characterization}. We then outline the basics of inverse semigroups and Billhardt's $\lambda$-semidirect products \cite{billhardt1992wreath} before finally discussing frames and Artin glueings.

\subsection{Weakly Schreier extensions}
The category $\Mon$ of monoids has zero morphisms between objects: the maps sending all elements to the identity. Consequently the kernel (dually the cokernel) of a morphism $f \colon A \to B$ can be defined in this category as the equaliser (dually the coequaliser) of $f$ and the zero morphism from $A$ into $B$. This allows us to define split extensions.

\begin{definition}
    For monoids $H$ and $N$, a \emph{split extension} of $H$ by $N$ is a diagram $\splitext{N}{k}{G}{e}{s}{H}$ in which $k$ is the kernel of $e$, $e$ is the cokernel of $k$, and $s$ is a section of $e$.
\end{definition}

Split extensions are often considered in the category of groups where it is well known that they correspond to semidirect products. In $\Mon$ this relationship is not so simple. The appropriate notion of an action in $\Mon$ gives rise to a monoidal notion of a semidirect product. However, as discussed above, these are only in bijection with what are known as \emph{Schreier split extensions}. %These are split extensions $\splitext{N}{k}{G}{e}{s}{H}$ in which for each element $g \in G$ there exists a unique $n \in N$ such that $g = k(n)se(g)$.

While this gives good motivation for these split extensions, one may ask if they are the only split extensions in $\Mon$ worth considering. As described in \cite{faul2019artin}, there is a related class of split extensions that arises naturally from the world of topology. That specific example will be discussed in \cref{subsec:Artin}, but for now we consider the class of split extensions in question.

\begin{definition}
    A split extension $\splitext{N}{k}{G}{e}{s}{H}$ is \emph{weakly Schreier} if for each $g \in G$ there exists $n \in N$ such that $g = k(n)se(g)$.
\end{definition}

Here, compared to Schreier extensions, we have dropped the uniqueness requirement on the $n \in N$ and this influences our ability to think of them as semidirect products.

Intuitively, when we have a split extension $\splitext{N}{k}{G}{e}{s}{H}$ of groups, there is a bijection between the set $N \times H$ and $G$ which sends the pair $(n,h)$ to $k(n)s(h)$. When we let the set $N \times H$ inherit a multiplication through this bijection we arrive at the associated semidirect product. This is why Schreier extensions are desirable in $\Mon$, as they ensure this bijection exists and allows the same inheritance of multiplication.

In both cases these semidirect products can be characterized by actions which encode this inherited multiplication. If $\alpha \colon H \times N \to N$ is an action of $H$ on $N$, then the multiplication is given by $(n,h) \cdot (n',h') = (n \alpha(h,n'),hh')$.

For weakly Schreier extensions there is no bijection between the elements of $N \times H$ and $G$, instead we have a surjective map sending $(n,h)$ to $k(n)s(h)$. We do, however, get a bijection between $G$ and the quotient of $N \times H$ given by this surjection. This quotient will then inherit a multiplication from $G$ and it is this object that we call the associated \emph{weak semidirect product}.

In terms of characterizing these weak semidirect products, it is no longer enough to just specify an action, as there is now this quotient to consider. In addition, the presence of this quotient means that we no longer need an action to specify the multiplication --- something weaker will do.

\begin{definition}
    For monoids $N$ and $H$, an equivalence relation on $N \times H$ is said to be \emph{admissible} if the following conditions hold.
    \begin{enumerate}
        \item $(n,1) \sim (n',1)$ implies $n = n'$,
        \item $(n,h) \sim (n',h')$ implies $h = h'$,
        \item $(n_1,h) \sim (n_2,h)$ implies $(xn_1,h) \sim (xn_2,h)$ for all $x \in N$,
        \item $(n_1,h) \sim (n_2,h)$ implies $(n_1,hy) \sim (n_2,hy)$ for all $y \in H$.
    \end{enumerate}
\end{definition}

Quotients arising from admissible equivalence relations are precisely the quotients that occur in weak semidirect products.

\begin{definition}
A function $\alpha \colon H \times N \to N$ is an \emph{action} compatible with an admissible equivalence relation $E$ on $N \times H$ if it satisfies the following conditions:

\begin{enumerate}
	\item $(n_1,h) \sim (n_2,h)$ implies $(n_1 \alpha(h,n),h) \sim (n_2 \alpha(h,n),h)$ for all $n \in N$,
	\item $(n,h') \sim (n',h')$ implies $(\alpha(h,n),hh') \sim (\alpha(h,n'),hh')$ for all $h \in H$,
	\item $(\alpha(h,nn'),h) \sim (\alpha(h,n)\cdot\alpha(h,n'),h)$,
	\item $(\alpha(hh',n),hh') \sim (\alpha(h,\alpha(h',n)),hh')$,
	\item $(\alpha(h,1),h) \sim (1,h)$,
	\item $(\alpha(1,n),1) \sim (n,1)$.
\end{enumerate}
\end{definition}

Conditions (3)--(6) are reminiscent of the usual action definition, except that in this case they are only required to hold up to equivalence.

These are essentially all the needed ingredients to characterize weak semidirect products. The only wrinkle is that two actions compatible with an admissible equivalence relation can sometimes induce the same multiplication on the quotient. Thus instead of considering the set $\Act_E(H,N)$ of actions compatible with $E$, we consider a quotient $\Act_E(H,N)/{\sim}$ where two actions $\alpha$ and $\beta$ are equivalent if $(\alpha(h,n),h) \sim (\beta(h,n),h)$ for all $n \in N$ and $h \in H$. See \cite{faul2019characterization} for more details.

\begin{proposition}\label{prop:equivtoschreier}
    For monoids $N$ and $H$, let $E$ be an admissible equivalence relation and $\alpha$ a compatible action. Then $(E,[\alpha])$ corresponds to a weakly Schreier extension
    \[
        \splitext{N}{k}{(N \times H/E, \cdot, [1,1])}{e}{s}{H}
    \]
    in which $k(n) = [n,1]$, $e([n,h]) = h$ and $s(h) = [1,h]$.
    Multiplication is defined as
    \[
        [n,h] \cdot [n',h'] = [n \alpha(h,n'),hh'].
    \]
\end{proposition}

Similarly we can consider a reverse of this process.
%% Mention q here!

\begin{proposition}\label{prop:schreiertoequiv}
    For monoids $N$ and $H$, let $\splitext{N}{k}{G}{e}{s}{H}$
    be a weakly Schreier extension and let $q \colon G \to N$ be a function such that $g = kq(g)se(g)$. Then we can associate to this a pair $(E,[\alpha])$ where $E$ is an admissible equivalence relation defined by
    \[
        (n,h) \sim (n',h') \iff k(n)s(h) = k(n')s(h').
    \]
    and $\alpha$ is a compatible action defined by
    \[
        \alpha(h,n) = q(s(h)k(n)).
    \]
\end{proposition}

Notice that there must exist such a map $q$ by virtue of the split extension being weakly Schreier (using the axiom of choice). The choice of $q$ does not end up mattering.

These two processes are inverses of one another up to isomorphism. Let us discuss what the morphisms in question are.

\begin{definition}
    A morphism of weakly Schreier extensions is a monoid homomorphism $f \colon G_1 \to G_2$ such that the three squares in the following diagram commute.
    
    \begin{center}
        \begin{tikzpicture}[node distance=1.5cm, auto]
            \node (A) {$N$};
            \node (B) [right of=A] {$G_1$};
            \node (C) [right of=B] {$H$};
            \node (D) [below of=A] {$N$};
            \node (E) [right of=D] {$G_2$};
            \node (F) [right of=E] {$H$};
            \draw[normalTail->] (A) to node {$k_1$} (B);
            \draw[transform canvas={yshift=0.5ex},-normalHead] (B) to node {$e_1$} (C);
            \draw[transform canvas={yshift=-0.5ex},->] (C) to node {$s_1$} (B);
            \draw[normalTail->] (D) to node {$k_2$} (E);
            \draw[transform canvas={yshift=0.5ex},-normalHead] (E) to node {$e_2$} (F);
            \draw[transform canvas={yshift=-0.5ex},->] (F) to node {$s_2$} (E);
            \draw[->] (B) to [swap] node {$f$} (E);
            \draw[double equal sign distance] (A) to (D);
            \draw[double equal sign distance] (C) to (F);
        \end{tikzpicture}
    \end{center}
\end{definition}

If $f \colon G_1 \to G_2$ is such a morphism, then for all $g \in G$, there exists an $n \in N$ such that $f(g) = f(k_1(n) \cdot s_1e_1(g)) = fk_1(n) \cdot fs_1(e_1(g)) = k_2(n) \cdot s_2(e_1(g))$. This means that any morphism between two weakly Schreier extensions must be unique.

\begin{proposition}
    The category $\WSExt(H,N)$ of weakly Schreier extensions between $N$ and $H$ and morphisms of weakly Schreier extensions is a preorder.
\end{proposition}

Inspired by the above we can define an order relation on our pairs $(E,[\alpha])$. Let $\splitext{N}{k_1}{G_1}{e_1}{s_1}{H}$ and $\splitext{N}{k_1}{G_1}{e_1}{s_1}{H}$ be weakly Schreier extensions and $f \colon G_1 \to G_2$ a morphism between them. Let $E_1$ and $E_2$ be the respective admissible equivalence relations and $q_1$ and $q_2$ associated Schreier retractions. Then our above calculation implies that $f([n,h]_{E_1}) = [n,h]_{E_2}$.
This will only be well-defined when $(n,h) \sim_{E_1} (n',h)$ implies $(n,h) \sim_{E_2} (n',h)$. The fact that $f$ must preserve multiplication is equivalent to the statement that $(q_1(s_1(h)(k_1(n))),h) \sim_{E_2} (q_2(s_2(h)k_2(n)),h)$.
Thus we can define the order as follows.

\begin{definition}
    Let $\WAct(H,N)$ have as objects pairs $(E,[\alpha])$ where $E$ is an admissible equivalence relation on $N \times H$ and $\alpha$ is a compatible action. Then we say $(E_1,[\alpha_1]) \leq (E_2,[\alpha_2])$ if and only if $(n,h) \sim_{E_1} (n',h)$ implies $(n,h) \sim_{E_2} (n',h)$ and for all $n \in N$ and $h \in H$ $(\alpha_1(h,n),h) \sim_{E_2} (\alpha_2(h,n),h)$.
\end{definition}

Using the transformations provided in \cref{prop:equivtoschreier} and \cref{prop:schreiertoequiv} we get the following equivalence.

\begin{theorem}
    The categories $\WSExt(H,N)$ and $\WAct(H,N)$ are equivalent.
\end{theorem}
%% Need to define the multiplication for later in the paper.
%% Seems like morphisms will need to be discussed before the characterization can be stated.

\subsection{Inverse semigroups and \texorpdfstring{$\lambda$}{lambda}-semidirect products}
%% Introduce inverse semigroups.
As discussed above, the standard semigroup semidirect product construction, when applied to two inverse semigroups, will not in general return an inverse semigroup. Thus, we study Billhardt's \emph{$\lambda$-semidirect product} \cite{billhardt1992wreath}.
The idea is to consider an algebraic structure on a subset of the product of two inverse semigroups.

\begin{definition}\label{def:action}
    Let $N$ and $H$ be inverse semigroups and let
    $\alpha \colon H \times N \to N$ be a function which we write as $\alpha(h,n) = h \cdot n$. Then $\alpha$ is an \emph{action of inverse semigroups} if the following conditions are satisfied for all $h,h' \in H$ and $n,n' \in N$.
    
    \begin{enumerate}
        \item $h \cdot (nn') = (h \cdot n)(h \cdot n')$,
        \item $hh' \cdot n = h \cdot (h' \cdot n)$.
    \end{enumerate}
\end{definition}

An action could, of course, equivalently be defined as a homomorphism from $H$ into the endomorphisms of $N$.

\begin{definition}\label{def:lambdaSemidirectProducts}
    Let $N$ and $H$ be inverse semigroups and let $H$ act on $N$. Then the $\lambda$-semidirect product associated to this action has as underlying set
    \[
        \{(n,h) \in N \times H : hh\inv \cdot n = n \}
    \]
    and multiplication defined by
    \[
        (n,h)(n',h') = \left(((h_1h_2)(h_1h_2)\inv \cdot n_1)(h_1 \cdot n'),h_1h_2\right)
    \]
\end{definition}

This multiplication resembles the multiplication of the standard semidirect product in a number of ways. The only disagreement is that instead of $(h \cdot n_2)$ being multiplied on the left by $n_1$, it is being multiplied on the left by $(h_1h_2)(h_1h_2)\inv \cdot n_1$.

%% Introduce Artin glueings and prove that they are weakly Schreier (which will only take a line or two). Refrain from expanding on this and demonstrating the quotient interpreation. In the next section, after proving \lambda-semidirect products are weakly Schreier, only then prove that Artin glueings are \lambda-semidirect products.
\subsection{Frames and Artin glueings}\label{subsec:Artin}
A motivating example of weakly Schreier extensions is the Artin glueing of frames. As we shall see in this section, Artin glueings have some interesting parallels to $\lambda$-semidirect products.

A frame is an algebraic structure that captures the lattice of open sets of a topological space. A frame has finite meet operations capturing finite intersections of open sets, and arbitrary joins corresponding to arbitrary unions of opens. Finally we require meets distribute over arbitrary joins. For a more comprehensive look at frames, see \cite{picado2011frames}.

\begin{definition}
    A frame $L$ is a poset with finite meets and arbitrary joins such that finite meets distribute over joins.
\end{definition}

We treat frames as algebraic structures and so the morphisms are just the maps preserving this structure.

\begin{definition}
    A morphism $f \colon L \to M$ of frames satisifies
    \begin{enumerate}
        \item $f(0) = 0$,
        \item $f(1) = 1$,
        \item $f(a \wedge b) = f(a) \wedge f(b)$,
        \item $f(\bigvee S) = \bigvee f(S)$.
    \end{enumerate}
\end{definition}

Given a continuous map between two topological spaces, we know that the preimage sends opens to opens and from set theoretic properties of the preimage, preserves the empty set, the whole space, finite intersections and arbitrary unions. That is, the preimage is a frame homomorphism between the corresponding lattices of open sets.

This idea gives rise to a contravariant functor from the category of topological spaces to the category of frames. Furthermore, the category of frames is easily seen to be a subcategory of the category of monoids under binary meet, though not a full one. Thus, we obtain a functor from the category of topological spaces into the category of monoids.

Transporting topological spaces into the category of monoids gives a worthwhile perspective on the well-known Artin glueing construction (see \cite{wraith1974glueing}). An Artin glueing of two topological spaces $N$ and $H$ is a topological space in which $N$ embeds as a closed subspace and $H$ as its open complement. For any two spaces there are in general many distinct Artin glueings and each is determined by a finite-meet-preserving map from $\O(H)$ to $\O(N)$. We present this construction below in the context of frames.

\begin{definition}
    Let $N$ and $H$ be frames and $f \colon H \to N$ be a (finite-)meet-preserving map. Then the Artin glueing $\Gl(f)$ is the frame of pairs $(n,h)$ for which $n \le f(h)$ with componentwise meets and joins.
\end{definition}

In \cite{faul2019artin} it was shown that Artin glueings precisely correspond to the weakly Schreier extensions in the full subcategory of $\Mon$ consisting of frames.

We see immediately some similarities with $\lambda$-semidirect products. The Artin glueing associates to a map an algebraic structure on a subobject of the product. Furthermore, the condition $n \le f(h)$ is equivalent to $n \wedge f(h) = n$. We will revisit this idea in \cref{section:artin-like-actions} where we will show that Artin glueings are in fact $\lambda$-semidirect products.

\section{\texorpdfstring{$\lambda$}{lambda}-Semidirect products of inverse monoids}\label{sec:lambdaSemidirectProductsOfInverseMonoids}
%% Haven't defined inverse semigroup
In order to relate the $\lambda$-semidirect product to weakly Schreier extensions of monoids, we must work inside the category of monoids. Thus, in this section we consider only \emph{inverse monoids} --- that is, inverse semigroups with a unit.

In order to consider $\lambda$-semidirect products in this context there is one standard modification that is made to the theory, relating to the definition of an action.

\begin{definition}\label{def:actionOfInverseMonoids}
    Let $N$ and $H$ be inverse monoids and let $\alpha \colon H \times N \to N$ be a function with application written $\alpha(h,n) = h \cdot n$. Then $\alpha$ is an \emph{action of inverse monoids} if it is an action of inverse semigroups and satisifies that for all $n \in N$
    \[
        1 \cdot n = n.
    \]
\end{definition}

Notably, it is not required that $h \cdot 1 = 1$. Thus, the action can equivalently be thought of as a monoid homomorphism into the monoid of semigroup endomorphisms of $N$.

%% TODO TODO Check that claim.
%% It appears true!
The $\lambda$-semidirect products we consider in this context are only taken with respect to actions of inverse monoids, as these are precisely the actions for which the associated $\lambda$-semidirect product is a monoid. (The pair $(1,1)$ acts as identity.)

\begin{proposition}\label{prop:kes}
    Let $N$ and $H$ be inverse monoids and let $\alpha \colon H \times N \to N$ be an action of inverse monoids. If $N \ltimes_\alpha H$ is the associated $\lambda$-semidirect product, then the following functions are monoid homomorphisms.
    
    \begin{enumerate}
        \item $k \colon N \to N \ltimes_\alpha H$, where $k(n) = (n,1)$,
        \item $e \colon N \ltimes_\alpha H \to H$, where $e(n,h) = h$,
        \item $s \colon H \to N \ltimes_\alpha H$, where $s(h) = (hh\inv \cdot 1,h)$.
    \end{enumerate}
\end{proposition}

\begin{proof}
    (1) We begin by proving that the function is well defined. This entails showing that $1(1\inv) \cdot n = n$. Since the inverse of $1$ is $1$ we use the fact that $\alpha$ is an action of inverse monoids.
    
    Next observe that
    \begin{align*}
        k(n_1)k(n_1)    &= (n_1,1)(n_2,1) \\
                    &= ((1(1\inv) \cdot n_1)(1 \cdot n_2),1) \\
                    &= (n_1n_2,1) \\
                    &= k(n_1n_2).
    \end{align*}
    It is clear the unit is preserved.
    
    (2) The function is automatically well defined and it is very easy to see that it preserves the multiplication and unit.
    
    (3) Again we begin by proving it is well defined. We must show that $(hh\inv) \cdot (hh\inv \cdot 1) = hh\inv \cdot 1$. This follows from the fact that $\alpha$ is action of semigroups and that $hh\inv$ is an idempotent.
    
    Finally observe the following calculation.
    \begin{align*}
        s(h_1)s(h_2)    &= (h_1h_1\inv \cdot 1, h_1)(h_2h_2\inv \cdot 1, h_2) \\
                    &= (((h_1h_2)(h_1h_2)\inv \cdot h_1h_1\inv \cdot 1)(h_1 \cdot h_2h_2\inv \cdot 1),h_1h_2) \\
                    &= ((h_1h_2h_2\inv h_1\inv h_1h_1\inv \cdot 1)(h_1h_2h_2\inv \cdot 1),h_1h_2) \\
                    &= (h_1h_2h_2\inv \cdot ((h_1\inv h_1h_1\inv \cdot 1)(1)),h_1h_2) \\
                    &= (h_1h_2h_2\inv h_1\inv \cdot 1, h_1h_2) \\
                    &= s(h_1h_2).
    \end{align*}
    Finally, note that $s(1) = (1(1\inv) \cdot 1, 1) = (1 \cdot 1, 1) = (1,1)$, the identity.
\end{proof}

It is apparent that $k$ is the kernel of $e$ and that $s$ splits $e$. Below we show that this diagram is indeed a weakly Schreier extension.

\begin{proposition}
    Let $N$ and $H$ be inverse monoids, $\alpha \colon H \times N \to N$ an action of inverse monoids, $N \ltimes_\alpha H$ the associated $\lambda$-semidirect product and let $k,e$ and $s$ be as in $\cref{prop:kes}$. Then $\splitext{N}{k}{N \ltimes_\alpha H}{e}{s}{H}$ is a weakly Schreier extension.
\end{proposition}

\begin{proof}
    As discussed, it is apparent that $k$ is the kernel and $s$ is the splitting of $e$. Thus, we must only show that $e$ is the cokernel of $k$ and that the weakly Schreier condition holds. We begin with the latter. Let $(n,h) \in N \ltimes_\alpha H$ and consider
    \begin{align*}
        k(n)s(h)    &= (n,1)(hh\inv \cdot 1, h) \\
                    &= ((hh\inv \cdot n)(1 \cdot hh\inv \cdot 1), h) \\
                    &= (hh\inv \cdot n, h) \\
                    &= (n,h).
    \end{align*}
    Here the last line follows because $(n,h)$ was assumed to belong to $S \ltimes_\alpha T$.
    
    To see that $e$ is the cokernel consider a map $t \colon N \ltimes_\alpha H \to X$ such that $tk$ is the zero morphism. We must show that there is a unique map $\ell \colon H \to X$ such that $t = \ell e$.
    
    By the above $t(n,h) = t(k(n)s(h)) = ts(h)$. We then need only observe that for $\ell = ts$ we have $\ell e(n,h) = ts(h)$, as required. Since $e$ has a splitting, it is epic and consequently the map $\ell = ts$ must be unique.
\end{proof}

Since $\lambda$-semidirect products of inverse monoids $N$ and $H$ are weakly Schreier extensions, we can view them instead as a particular admissible quotient of $N \times H$ paired with a compatible action.

\subsection{The admissible quotient and compatible action}

Let $\alpha$ be an action of inverse monoids of $H$ on $N$ and let $\splitext{N}{k}{N \ltimes_\alpha H}{e}{s}{H}$ be the weakly Schreier extension corresponding to the associated $\lambda$-semidirect product. Then two pairs $(n_1,h_1)$ and $(n_1,h_2)$ will be related in the admissible quotient if and only if $k(n_1)s(h_1) = k(n_2)s(h_2)$. This amounts to requiring that $h_1 = h_2$ and that $h_1h_1\inv \cdot n_1 = h_1h_1\inv \cdot n_2$.

\begin{proposition}
    Let $\splitext{N}{k}{N \ltimes_\alpha H}{e}{s}{H}$ be the weakly Schreier extension corresponding to a $\lambda$-semidirect product. Then $(n,h) \sim (hh\inv \cdot n, h)$ in the associated admissible equivalence relation.
\end{proposition}

\begin{proof}
    Since the second components agree, we need only verify that $hh\inv \cdot n = hh\inv \cdot hh\inv \cdot n$. This follows from $\alpha$ being an action of semigroups and from the idempotence of $hh\inv$.
\end{proof}

This means that each equivalence class $[n,h]$ has a canonical representative $(hh\inv \cdot n,h)$. The set of these representatives is easily seen to be the underlying set of $S \ltimes_\alpha T$.

In order to determine a compatible action we first consider the associated Schreier retraction. It is easy to see that the first projection $\pi_1 \colon N \ltimes_\alpha H \to N$ is such a map. (Recall that the Schreier retraction need not be monoid homomorphisms). Given this Schreier retraction the compatible action is thus $\beta \colon H \times N \to N$ where
\begin{align*}
    \beta(h,n)  &= \pi_1(s(h)k(n)) \\
                &= \pi_1((hh\inv \cdot 1, h)(n,1)) \\
                &= \pi_1((hh\inv \cdot hh\inv \cdot 1)(h \cdot n),h) \\
                &= (hh\inv \cdot 1)(h \cdot n) \\
                &= (hh\inv \cdot 1)(hh\inv h \cdot n) \\
                &= hh\inv \cdot (1(h\cdot n)) \\
                &= h \cdot n.
\end{align*}

Thus, the compatible action $\beta$ is just the original action $\alpha$.

Recall that from the weakly Schreier perspective the multiplication is given by
\[
    [n_1,h_1][n_2,h_2] = [n_1(h_1 \cdot n_1),h_1h_2]
\]

The element $(n_1(h_1 \cdot n_2),h_1h_2)$ will not in general be the canonical element of its class. Thus, we canonicalise it and arrive at
\begin{align*}
    (h_1h_2(h_1h_2)\inv \cdot (n_1(h_1 \cdot n_2), h_1h_2)
        &= ((h_1h_2(h_1h_2)\inv \cdot n_1)(h_1h_2(h_1h_2)\inv \cdot h_1 \cdot n_2),h_1h_2) \\
        &= ((h_1h_2(h_1h_2)\inv \cdot n_1)(h_1 \cdot (h_2h_2\inv \cdot n_2)),h_1h_2).
\end{align*}

Note that if $(n_2,h_2) \in N \ltimes_\alpha H$, then the expression reduces to $((h_1h_2(h_1h_2)\inv \cdot n_1)(h_1 \cdot n_2),h_1h_2)$ which is precisely the multiplication of $N \ltimes_\alpha H$.

\section{The preorder of $\lambda$-semidirect products}

Since the set of weakly Schreier extensions between monoids $N$ and $H$ has a natural preorder structure, we can now ask what order this induces on the set of $\lambda$-semidirect products when we take $N$ and $H$ to be inverse monoids.

It will be convenient to think in terms of the actions of inverse monoids instead of the $\lambda$-semidirect products themselves. Thus, we consider the preorder induced on the set of actions by the function sending an action to its associated weakly Schreier extension.

This function is not injective as two distinct actions can be mapped to isomorphic weakly Schreier extensions.

%% Perhaps this should be reformulated into a proposition. If I do I should name the function. I must emph 'not'
%% TODO TODO Check this!
\begin{example}
    Let $N$ be an inverse monoid with at least two distinct idempotents $u$ and $u'$ and let $H$ be an inverse semigroup satisfying that $h_1h_2 = 1$ implies $h_1 = 1 = h_2$.
    
    Consider the function $\alpha_u \colon H \times N \to N$ where $\alpha_u(h,n) = u$ whenever $h \ne 1$ and $\alpha_u(1,n) = n$. Because $h_1h_2 = 1$ implies $h_1 = 1 = h_2$ we have that $\alpha_u$ is an action of inverse monoids.
    
    Similarly, consider the action $\alpha_{u'} \colon H \times N \to N$ where $\alpha_{u'}(h,n) = u'$ whenever $h \ne 1$ and $\alpha_{u'}(1,n) = n$.
    
    It is apparent that $\alpha_u \ne \alpha_{u'}$. Furthermore, both actions result in an equivalence relation in which $(n_1,h) \sim (n_2,h)$ for all $n_1,n_2 \in N$ and $h \in H-\{1\}$, and $(n_1,1) \sim (n_2,1)$ if and only if $n_1 = n_2$. The multiplications agree as required, as in both equivalence relations we have that $(\alpha_u(h,n),h) \sim (\alpha_{u'}(h,n),h)$.
\end{example}

\begin{proposition}
    Let $N$ and $H$ be inverse monoids and let $\alpha \colon H \times N \to N$ and $\beta \colon H \times N \to N$ be actions of inverse monoids. Then $\alpha \le \beta$ if and only if for all $n \in N$ and $h \in H$, $\beta(hh\inv,\alpha(h,n)) = \beta(h,n)$.
\end{proposition}

\begin{proof}
    ($\Rightarrow$) Suppose that $\alpha \le \beta$.
    
    Then $(\alpha(h,n),h) \sim_\beta (\beta(h,n),h)$. Unwinding this gives
    \begin{align*}
        \beta(hh\inv,\alpha(h,n))   &= \beta(hh\inv, \beta(h,n)) \\
                                    &= \beta(h,n).
    \end{align*}
    
    ($\Leftarrow$) Suppose that for all $n \in N$ and $h \in H$, $\beta(hh\inv,\alpha(h,n)) = \beta(h,n)$.
    
    First we show that $(n_1,h) \sim_\alpha (n_2,h)$ implies $(n_1,h) \sim_\beta (n_2,h)$. Suppose that $(n_1,h) \sim_\alpha (n_2,h)$. This means that $\alpha(hh\inv,n_1) = \alpha(hh\inv,n_2)$. Thus, making use of our assumption we find
    \begin{align*}
        \beta(hh\inv, n_1)  &= \beta(hh\inv(hh\inv)\inv, \alpha(hh\inv,n_1)) \\
                            &= \beta(hh\inv, \alpha(hh\inv,n_1)) \\
                            &= \beta(hh\inv, \alpha(hh\inv,n_2)) \\
                            &= \beta(hh\inv, n_2).
    \end{align*}
    
    Now we must show that $(\alpha(h,n),h) \sim_\beta (\beta(h,n),h)$. For these to be related we need that $\beta(hh\inv,\alpha(h,n)) = \beta(hh\inv, \beta(h,n))$. By assumption $\beta(hh\inv,\alpha(h,n)) = \beta(h,n)$ and combined with the fact that $\beta(h,n) = \beta(hh\inv,\beta(h,n))$, we obtain the desired equality.
\end{proof}

%% Maybe some sentence discussing the intuition for this condition?
%% TODO TODO Provide an example here which is not Artin glueing like.

\section{Artin-glueing-like actions}\label{section:artin-like-actions}

Given an order structure on the set of $\lambda$-semidirect products, it is natural to consider if meets and joins exist. In the section we show that joins exist for a natural class of $\lambda$-semidirect products, reminiscent of Artin glueings of frames.

As alluded to in the introduction, Artin glueings of frames are nothing more than a certain class of $\lambda$-semidirect products between certain meet-semilattices.

\begin{proposition}
    Let $N$ and $H$ be frames considered in the category of monoids and let $f \colon H \to N$ be a monoid homomorphism. Then the Artin glueing $\Gl(f)$ is a $\lambda$-semidirect product of $H$ by $N$.
\end{proposition}

\begin{proof}
    The action corresponding to $\Gl(f)$ is given by $\alpha(h,n) = f(h) \wedge n$. Let us begin by confirming that this is an action of inverse monoids.
    
    It is clear that $\alpha(1,n) = n$ as $f$ preserves the identity. Next observe
    \begin{align*}
        \alpha(h,n \wedge n')   &= f(h) \wedge n \wedge n' \\
                                &= f(h) \wedge n \wedge f(h) \wedge n' \\
                                &= \alpha(h,n) \wedge \alpha(h,n').
    \end{align*}
    
    Finally consider
    \begin{align*}
        \alpha(h \wedge h', n)  &= f(h \wedge h') \wedge n \\
                                &= f(h) \wedge f(h') \wedge n \\
                                &= \alpha(h, f(h') \wedge n) \\
                                &= \alpha(h, \alpha(h',n)).
    \end{align*}
    
    Thus, it remains only to show that $N \ltimes_\alpha H = \Gl(f)$.
    
    Since the inverse of an element in a meet semilattice is itself and because of idempotence, we have that the elements of $N \ltimes_\alpha H$ are those pairs $(n,h)$ in which $n = f(h) \wedge n$. These are precisely the pairs in which $n \le f(h)$ and so $N \ltimes_\alpha H$ and $\Gl(f)$ agree on elements.
    
    Using the same properties of meet-semilattices we see that the multiplication in $N \ltimes_\alpha H$ is given by
    \begin{align*}
        (n,h)(n',h')    &= ((f(h) \wedge f(h') \wedge n) \wedge (f(h) \wedge n'), h \wedge h') \\
                        &= (f(h) \wedge n \wedge f(h') \wedge n', h \wedge h') \\
                        &= (n \wedge n', h \wedge h').
    \end{align*}
    
    This coincides with the multiplication of $\Gl(f)$ and so we are done.
\end{proof}

As discussed in \cite{faul2019artin}, if $N$ and $H$ are frames and $f,g \colon H \to N$ are monoid homomorphisms, then $\Gl(f \wedge g)$ is the join of $\Gl(f)$ and $\Gl(g)$ in the order structure on Artin glueings. In fact, as we shall see, $\Gl(f \wedge g)$ is the join of $\Gl(f)$ and $\Gl(g)$ in $\WSExt(H,N)$.

Inspired by the above, we would like to consider actions $\alpha$ of inverse monoids such that $\alpha(h,n) = f(h) \cdot n$ where $f$ is some function from $H$ into $N$. The condition that $\alpha$ be an action precludes many functions $f$ from serving this purpose. It is sufficient for $f$ to factor through the central idempotents of $S$.

\begin{proposition}\label{prop:artin-like-action}
    Let $H$ and $N$ be inverse monoids and let $f \colon H \to E(N) \cap Z(N)$ be a monoid homomorphism into the central idempotents of $N$, where $E(N)$ denotes the idempotents of $N$ and $Z(N)$ the central elements. Then $\alpha(h,n) = f(h) \cdot n$ is an action of inverse monoids.
\end{proposition}

\begin{proof}
    For $\alpha(h,n_1n_2)$ we have
    \begin{align*}
        \alpha(h,n_1n_2)    &= f(h) \cdot n_1n_2 \\
                            &= f(h)f(h) \cdot n_1n_2 \\
                            &= f(h)n_1 \cdot f(h)n_2 \\
                            &= \alpha(h,n_1)\alpha(h,n_2),
    \end{align*}
    which makes use of the fact that $f(h)$ is a central idempotent.
    
    Next we must check that $\alpha(h_1h_2,n) = \alpha(h_1,\alpha(h_2,n))$. Here we consider
    \begin{align*}
        \alpha(h_1h_2,n)    &= f(h_1h_2) \cdot n \\
                            &= f(h_1)f(h_2) \cdot n \\
                            &= f(h_1) \cdot \alpha(h_2,n) \\
                            &= \alpha(h_1, \alpha(h_2,n)).
    \end{align*}
    
    The final condition follows easily with $\alpha(1,n) = f(1) \cdot n = 1 \cdot n = n$.
\end{proof}

\begin{definition}
    Let $H$ and $N$ be inverse monoids and $f \colon H \to E(N) \cap Z(N)$ a monoid homomorphism into the central idempotents of $N$. Then we call the action $\alpha_f(h,n) = f(h) \cdot n$ the \emph{Artin-like-action} corresponding to $f$.
\end{definition}

The $\lambda$-semidirect products resulting from Artin-like-actions have many nice properties. For instance, when interpreted as a weakly Schreier extensions, the canonical element of each equivalence class can be easily seen to be the smallest element in each class.

Furthermore, just as in the frame setting, we can combine two actions of this form in a natural way.

\begin{proposition}
    Let $N$ and $H$ be inverse semigroups and let $\alpha_f, \alpha_g$ be Artin-like-actions corresponding to the maps $f,g \colon H \to E(N) \cap Z(N)$ respectively. Then the action $\gamma \colon H \times N \to N$ given by $\gamma(h,n) = f(h)g(h)n$, is an Artin-like-action.
\end{proposition}

\begin{proof}
    We claim that $\gamma$ corresponds to $\alpha_{f \cdot g}$, where $f \cdot g(h) = f(h)g(h)$. It is clear that $f\cdot g$ preserves the identity. To see that it preserves multiplication we make use of the fact that both $f$ and $g$ map into the centre of $N$. Thus we have
    \begin{align*}
        f \cdot g(h_1h_2)  &= f(h_1h_2)g(h_1h_2) \\
                &= f(h_1)f(h_2)g(h_1)g(h_2) \\
                &= f(h_1)g(h_1)f(h_2)g(h_2) \\
                &= f \cdot g(h_1)f \cdot g(h_2)
    \end{align*}
    and can conclude that $f \cdot g$ is a monoid homomorphism as required.
    
    We then invoke \cref{prop:artin-like-action} and we are done.
\end{proof}

%% TODO TODO Make uniform the noation for discussing these actions when they occur in \WSExt.
\begin{proposition}
    Let $N$ and $H$ be inverse monoids and let $f,g \colon H \to E(N) \cap Z(N)$ be monoid homomorphisms into the central idempotents of $N$. Then the join of $\alpha_f$ and $\alpha_g$ exists in $\WSExt(H,N)$ and is equal to $\alpha_{f \cdot g}$.
\end{proposition}

\begin{proof}
    First we show that $\alpha_{f \cdot g}$ is larger than $\alpha_f$ and $\alpha_g$ in $\WSExt(H,N)$.
    
    If $(n_1,h) \sim_{\alpha_f} (n_2,h)$ then $f(h)n_1 = f(h)n_2$. Thus, $g(h)f(h)n_1 = g(h)f(h)n_2$ and since $g(h)$ is central, we have $fg(h)n_1 = fg(h)n_2$. This means that $(n_1,h) \sim_{\alpha_{f \cdot g}} (n_2,h)$ as required. This same argument gives that $(n_1,h) \sim_{\alpha_g} (n_2,h)$ implies that $(n_1,h) \sim_{\alpha_{f \cdot g}} (n_2,h)$.
    
    The final condition to check is that $(g(h)n,h) \sim_{\alpha_{f \cdot g}} (fg(h)n,h) \sim_{\alpha_{f \cdot g}} (f(h)n,h)$. This follows because $f(h)$ and $g(h)$ are both central and idempotent.
    
    To show that $\alpha_{f \cdot g}$ is the join suppose we have a weakly Schreier extension $(E,\beta)$ larger than $\alpha_f$ and $\alpha_g$, but smaller than $\alpha_{f \cdot g}$. Since $(E, \beta)$ is smaller than $\alpha_{f \cdot g}$, we have that if $(n_1,h) \sim_E (n_2,h)$ then $(n_1,h) \sim_{\alpha_{f \cdot g}} (n_2,h)$. We will show that $(E,\beta)$ being larger than $\alpha_f$ and $\alpha_g$
    means that $(n_1,h) \sim_{\alpha_{f \cdot g}} (n_2,h)$ implies that $(n_1,h) \sim_{E} (n_2,h)$.
    
    We know that $(g(h)n,h) \sim_E (n,h) \sim_E (f(h)n,h)$ for all $n \in N$ and $h \in H$. Now suppose that $(n_1,h) \sim_{\alpha_{f \cdot g}} (n_2,h)$. This means that $f(h)g(h)n_1 = f(h)g(h)n_2$. Now simply consider
    \begin{align*}
        (n,h)   &\sim_E (f(h)n_1,h) \\
                &\sim_E (f(h),1)(n_1,h) \\
                &\sim_E (f(h),1)(g(h)n_1,h) \\
                &\sim_E (f(h)g(h)n_1,h) \\
                &\sim_E (f(h)g(h)n_2,h) \\
                &\sim_E (n_2,h).
    \end{align*}
    
    Thus the equivalence relations are equal and so $(E,\beta) = \alpha_{f \cdot g}$.
\end{proof}

Notice that this gives that $\Gl(f \wedge g) = \Gl(f) \vee \Gl(g)$ in $\WSExt(H,N)$.

%% Can Artin-like actions be considered for arbitrary monoids?

%% Check whether S is closed in such an Artin-like-action and that T is open.

%% This proof may benefit from a general result that shows that we only need to worry about there being a finer quotient as its not possible for the multiplication to be different.

%% TODO TODO Establish that we will omit most of the extension paraphenalia and only use the middle term generally.

%% The fact that Artin glueings are examples of semidirect products on these weaker actions (where 1 need not be sent to 1) strongly motivates their study outside the context of semigroup theory.

\bibliographystyle{abbrv}
\bibliography{bibliography}
\end{document}